\documentclass[10pt,reqno]{amsart}
\usepackage{amsmath,amscd,amssymb}
\usepackage[colorlinks,citecolor=blue,urlcolor=blue]{hyperref}
\usepackage{color}
\usepackage{tikz}

\theoremstyle{plain}
   \newtheorem{theorem}{Theorem}[section]
   
   \newtheorem{lemma}[theorem]{Lemma}
   \newtheorem{corollary}[theorem]{Corollary}

\theoremstyle{definition}
   \newtheorem{definition}{Definition}[section]
   
   \newtheorem{question}{Question}[section]
   \newtheorem{example}{Example}[section] 
\theoremstyle{remark}
 \newtheorem{remark}{Remark}[section]

\newcommand{\R}{\mathbb{R}}
\newcommand{\Proj}{\mathbb{P}}
\newcommand{\Z}{\mathbb{Z}}

\def\newop#1{\expandafter\def\csname #1\endcsname{\mathop{\rm
#1}\nolimits}}

\newop{Ehr}
\newop{vol}
\newop{Vol}
\newop{conv}
\newop{Supp}
\newop{supp}
\newop{aw-height}
\newop{Nasc}
\newop{nasc}
\newop{comp}
\newop{Des}
\newop{des}
\newop{maxDes}
\newop{awheight}
\newop{lcm}
\newop{red}

\keywords{Box polynomial, local $h^\ast$-polynomial, Eulerian polynomial, Ehrhart theory, simplex, weighted projective space, numeral systems, simplices for numeral systems, factoradics, real-rooted, unimodal, symmetric, log-concave}

\begin{document}

\title{Local $h^\ast$-Polynomials of Some Weighted Projective Spaces}

\author{Liam Solus}
\date{\today}
\address{Institutionen f\"or Matematik, KTH, SE-100 44 Stockholm, Sweden}
\email{solus@kth.se}

\begin{abstract}
There is currently a growing interest in understanding which lattice simplices have unimodal local $h^\ast$-polynomials (sometimes called box polynomials); specifically in light of their potential applications to unimodality questions for Ehrhart $h^\ast$-polynomials. 
In this note, we compute a general form for the local $h^\ast$-polynomial of a well-studied family of lattice simplices whose associated toric varieties are weighted projective spaces.  
We then apply this formula to prove that certain such lattice simplices, whose combinatorics are naturally encoded using common systems of numeration, all have real-rooted, and thus unimodal, local $h^\ast$-polynomials.  
As a consequence, we discover a new restricted Eulerian polynomial that is real-rooted, symmetric, and admits intriguing number theoretic properties.
\end{abstract}

%

\maketitle
\thispagestyle{empty}

\section{Introduction}
\label{sec: introduction}
For a positive integer $n\in\Z_{\geq0}$ we set $[n]:=\{1,\ldots,n\}$ and $[n]_0:=\{0,\ldots,n\}$.
Let $p(z) = p_0+p_1z+\cdots+p_nz^n$ be the generating polynomial for a combinatorial sequence $(p_0,\ldots,p_n)$; in particular, assume that $p_k$ is a nonnegative integer for all $k\in[n]_0$.  
A popular endeavour in combinatorics is to understand the distributional properties of the polynomial $p(z)$, and thus of the sequence $(p_0,\ldots,p_n)$.  
We say that $p(z)$ is {\bf unimodal} if $p_0\leq p_1\leq \cdots\leq p_t\geq \cdots\geq p_{n-1}\geq p_n$ for some $t\in[n]_0$. 
We say it is {\bf log-concave} if $p_i^2\geq p_{i-1}p_{i+1}$ for all $i\in[n-1]$, and we say it is {\bf symmetric} with respect to $m\in\Z_{\geq0}$ if $p_i = p_{m-i}$ for all $i\in[m]_0$.  
As stated, these properties seem mysteriously independent of the generating polynomial $p(z)$, and only really pertaining to the combinatorial sequence $(p_0,\ldots,p_n)$.  
However, the structure of the polynomial $p(z)$, and particularly it roots, can help us deduce these various distributional properties.  
We say that $p(z)$ is {\bf real-rooted} if all of its roots are real numbers.  
If $p(z)$ is real-rooted then it is also log-concave and unimodal \cite[Theorem 1.2.1]{B89}.  
Thus, it is particularly desirable if the generating polynomial $p(z)$ is both real-rooted and symmetric.  

In algebraic combinatorics, a well-studied family of generating polynomials are the (Ehrhart) $h^\ast$-polynomials of lattice polytopes.  
Let $P\subset\R^n$ be a $d$-dimensional convex lattice polytope; i.e.~a convex polytope whose affine span has dimension $d$ and all of whose vertices lie in $\Z^n$.  
For a nonnegative integer $t\in\Z_{\geq0}$, the {\bf $t^{th}$ dilate of $P$} is $tP:=\{tp\in\R^n \, : \, p\in P\}$, and the {\bf Ehrhart series of $P$} is
\[
\Ehr_P(z):=\sum_{t\geq0}|tP\cap\Z^n|z^t = \frac{h_0^\ast+h_1^\ast z+\cdots+h_d^\ast z^d}{(1-z)^{d+1}}.
\]
When written in its closed rational form, $\Ehr_P(z)$ appears as in the right-most expression above.
The polynomial in the numerator is called the {\bf (Ehrhart) $h^\ast$-polynomial} of $P$, and it is known to have only nonnegative integer coefficients \cite{S80}.
In recent decades, the distributional properties of $h^\ast$-polynomials have been investigated extensively (see for instance~\cite{B16}).  

In the case of lattice simplices, the $h^\ast$-polynomial has a closely related local-invariant known as its local $h^\ast$-polynomial \cite{KS16}, or box polynomial \cite{BR07,B16,SV13}.
Let 
\[
\Delta :=\conv(v^{(0)},\ldots,v^{(d)})\subset\R^n
\]
be a lattice $d$-simplex defined as the convex hull of the $d+1$ affinely independent points $v^{(0)},\ldots,v^{(d)}\in\Z^n$.  
The {\bf open parallelpiped} of $\Delta$ is 
\[
\Pi_\Delta^\circ :=
\left\{
\sum_{i=0}^d\lambda_i(v^{(i)},1)\in\R^{n+1}
\, : \,
0<\lambda_i<1, \,  i\in[d]_0
\right\},
\]
and the {\bf local $h^\ast$-polynomial}, or {\bf box polynomial}, of $\Delta$ is 
\[
\ell^\ast(\Delta;z) := \sum_{(x_1,\ldots,x_{n+1})\in\Pi_\Delta^\circ\cap\Z^n}z^{x_{n+1}}.
\]
We additionally define the {\bf half-open parallelpiped} of $\Delta$, denoted $\Pi_\Delta$, by replacing the inequalities $0<\lambda_i$ in the definition of $\Pi_\Delta^\circ$ with $0\leq\lambda_i$ for all $i\in[d]_0$.
As outlined in~\cite{B16} and~\cite{SV13} and applied in~\cite{GS18}, unimodality of local $h^\ast$-polynomials of lattice simplices can sometimes be used to recover unimodality of $h^\ast$-polynomials by way of a theorem shown in~\cite{BM85}.  
Moreover, as demonstrated in~\cite{GS18}, the distributional properties of local $h^\ast$-polynomials can also be used to answer questions on the distributional properties of a related family of generating polynomials; namely, the local $h$-polynomials of subdivisions of simplices \cite{A16}.
The applications of local $h^\ast$-polynomials to unimodality of $h^\ast$-polynomials prompted the following general question, posed in~\cite{B16}:
\begin{question}
\cite[Question 5]{B16}
\label{quest: box polynomials}
Which lattice simplices have unimodal local $h^\ast$-polynomials?
\end{question}
It can be seen from the definition of $\ell^\ast(\Delta;z)$ that local $h^\ast$-polynomials are always symmetric with respect to $d+1$.  
Thus, if we can further show $\ell^\ast(\Delta;z)$ is real-rooted, it will have all of the desirable distributional properties we previously outlined, including the unimodality requested by Question~\ref{quest: box polynomials}.  

One well-studied family of simplices for which we may consider Question~\ref{quest: box polynomials} are those whose associated toric varieties are {\bf weighted projective spaces} \cite{C02}.
For convenience we may refer to a simplex in this family simply as a weighted projective space.
Recently, the distributional properties of the $h^\ast$-polynomials of a particularly nice family of weighted projective spaces, denoted $\mathcal{Q}$, have been intensely investigated \cite{BD16,BDS16,BL18,LS18,S17}.  
Within the family $\mathcal{Q}$ live the {\bf simplices for numeral systems}, examples of which were shown in~\cite{S17} to have real-rooted $h^\ast$-polynomials.   
In this note, we first state a general formula for the local $h^\ast$-polynomials for each simplex in $\mathcal{Q}$.  
We then use this formula to show that all simplices for numeral systems studied in~\cite{S17} also have real-rooted local $h^\ast$-polynomials, thereby answering Question~\ref{quest: box polynomials} for this family of weighted projective spaces.

In Section~\ref{sec: some weighted projective spaces}, we recall the definition of the weighted projective spaces $\mathcal{Q}$ and compute a general formula for their associated local $h^\ast$-polynomials.
We additionally recall some necessary preliminaries on numeral systems \cite{F85} and real-rooted polynomials in subsections~\ref{subsec: numeral systems and weighted projective spaces} and~\ref{subsec: real-rootedness}, respectively.  
In Section~\ref{sec: the factoradic simplex}, we prove that the local $h^\ast$-polynomial of the weighted projective spaces associated to the factoradic numeral system are all real-rooted, and thus unimodal.
In doing so, we discover a previously unknown restricted version of the Eulerian polynomial that is both symmetric and real-rooted.    
In Section~\ref{sec: the base-r simplex}, we prove that the local $h^\ast$-polynomials of the base-$r$-simplices for numeral systems are all real-rooted as well.  
In Section~\ref{sec: final remarks}, we then outline some future directions motivated by this work pertaining to $\gamma$-nonnegativity and local $h$-polynomials of subdivisions of simplicial complexes.

\section{Some Weighted Projective Spaces}
\label{sec: some weighted projective spaces}

Let $e^{(1)},\ldots,e^{(n)}$ denote the standard basis vectors in $\R^n$.
Given a weakly increasing vector of positive integers $q :=(q_1,\ldots,q_n)$, define an $n$-simplex $\Delta_{(1,q)}$ as the convex hull
\[
\Delta_{(1,q)} :=\conv\left(e^{(1)},\ldots,e^{(n)},-\sum_{i=1}^nq_ie^{(i)}\right)\subset\R^n.
\]
The projective toric variety associated to $\Delta_{(1,q)}$ is the weighted projective space in which the weight $q_i$ is assigned to the $i^{th}$ coordinate \cite{C02}.  
Recently, the simplices $\Delta_{(1,q)}$ have been considered for their combinatorial properties, particularly in regards to the unimodality of their $h^\ast$-polynomials \cite{BD16,BDS16,BL18,LS18,S17}.  
Thus, in light of Question~\ref{quest: box polynomials}, it is natural to ask which $\Delta_{(1,q)}$ have unimodal local $h^\ast$-polynomials.  
To this end, we first determine a general formula for the local $h^\ast$-polynomial of $\Delta_{(1,q)}$.  

\begin{theorem}
\label{thm: box polynomial of q}
Let $q := (q_1,\ldots,q_n)$ be a sequence of positive integers, set $Q := 1+q_1+\cdots+q_n$, and define the set 
\[
T_q := 
\left\{
b\in[Q-1]
\, : \,
Q \nmid\,q_ib \mbox{ for all $i\in[n]$}
\right\}.
\]
Then the local $h^\ast$-polynomial of the $n$-simplex $\Delta_{(1,q)}$ is 
\[
\ell^\ast(\Delta_{(1,q)};z) = \sum_{b\in T_q}z^{\omega(b)},
\qquad
\mbox{where}
\qquad
\omega(b) = b - \sum_{i=1}^n\left\lfloor\frac{q_ib}{Q}\right\rfloor.
\]
\end{theorem}

\begin{proof}
In the following we let $\Pi_q:=\Pi_{\Delta_{(1,q)}}$ denote the half-open parallelpiped of $\Delta_{(1,q)}$, and we let $\Pi_q^\circ:=\Pi_{\Delta_{(1,q)}}^\circ$ denote the associated open parallelpiped.  
To prove the desired result, we take an analogous approach to the proof of Theorem 2.2 in~\cite{BDS16}, in which the $h^\ast$-polynomial of $\Delta_{(1,q)}$ is computed by way of the heights of the lattice points in $\Pi_q$.  
In this fashion, we let $(v^{(1)},1)$, $(v^{(2)},1)$, \ldots, $(v^{(n)},1)$, $(v^{(0)},1)$ denote the columns of the matrix
\[
\begin{pmatrix}
1		&	1		&	1		&	\cdots	&	1		&	1		\\
1		&	0		&	0		&	\cdots	&	0		&	-q_1		\\
0		&	1		&	0		&	\cdots	&	0		&	-q_2		\\
0		&	0		&	1		&	\cdots	&	0		&	-q_3		\\
\vdots	&	\vdots	&	\vdots	&	\ddots	&	\vdots	&	\vdots	\\
0		&	0		&	0		&	\cdots	&	1		&	-q_n		\\
\end{pmatrix}
\]
as read from left-to-right.  
Then every lattice point in the open parallelpiped $\Pi_q^\circ$ is of the form 
\[
p = \sum_{i=0}^n\lambda_i(v^{(i)},1),
\]
where $0<\lambda_i <1$ for all $0\leq i \leq n$.  
Moreover, by Proposition 4.4 in~\cite{N07} and Cramer's Rule we know that
$
\lambda_i = \frac{b_i}{Q},
$
for some $0\leq \lambda_i <Q$.  
However, since we are assuming $p\in\Z^{n+1}$, then for $1\leq i\leq n$ it must be that 
$
\lambda_i = q_i\lambda_0-\left\lfloor q_i\lambda_0\right\rfloor.  
$
So it follows that $p\in\Pi_q^\circ$ if and only if $0<\lambda_0<1$ and 
$
0<q_i\lambda_0-\left\lfloor q_i\lambda_0\right\rfloor<1
$
for all $1\leq i \leq n$.  
Since $\lambda_0 = \frac{b}{Q}$ for some $0\leq b<Q$, this is equivalently stated as $p\in\Pi_q^\circ$ if and only if $Q\nmid\,q_ib$ for all $i\in[n]$ and $0<b<Q$.  
Since, for each $0<b<Q$, the height of the corresponding lattice point in $\Pi_q^\circ$ is given by
\[
\sum_{i=0}^n\lambda_i = b-\sum_{i=1}^n\left\lfloor\frac{q_ib}{Q}\right\rfloor,
\]
the proof is complete.
\end{proof}

As a first example, we can use Theorem~\ref{thm: box polynomial of q} to compute the local $h^\ast$-polynomial for the $n$-simplex whose associated projective toric variety is $\Proj^n$; i.e.~$n$-dimensional projective space.  
\begin{corollary}
\label{cor: projective space}
Let $\Delta_{\left(1,q^{(n)}\right)}$ be the $n$-simplex with $q^{(n)} := (1,1,\ldots,1)\in\R^n$.  
Then the local $h^\ast$-polynomial of $\Delta_{\left(1,q^{(n)}\right)}$ is
\[
\ell^\ast(\Delta_{\left(1,q^{(n)}\right)};z) = z+z^2+\cdots+z^{n} = zh^\ast(\Delta_{\left(1,q^{(n-1)}\right)};z).
\]
\end{corollary}

\begin{proof}
Since $Q = n+1$ and $q^{(n)}_i = 1$ for all $1\leq i\leq n$ for $\Delta_{\left(1,q^{(n)}\right)}$, then $b\in[Q-1]$ is in the set $T_{q^{(n)}}$ if and only if $n+1$ divides $b$.  
Therefore, $T_{q^{(n)}} = [Q-1]$, and the first inequality follows by checking that $\omega(b) = b$ for all $b\in[Q-1]$.  
The second inequality follows from well-known results.  
\end{proof}
In a recent paper \cite{GS18}, the authors showed that the local $h^\ast$-polynomials of the {\bf $s$-lecture hall simplices} are all real-rooted, and therefore affirmatively answered Question~\ref{quest: box polynomials} for this family of simplices.  
They additionally showed that they contain the local $h$-polynomials for some well-studied subdivisions of simplices which are closely related to derangement polynomials.
The $h^\ast$-polynomials of the $s$-lecture hall simplices, called the {\bf $s$-Eulerian polynomials}, are also real-rooted and unimodal \cite{SV15}.
The results of~\cite{GS18} then suggest that these nice properties of the $h^\ast$-polynomial of a simplex can be inherited by its local $h^\ast$-polynomial. 
In~\cite{S17} the author indexed families of $\Delta_{(1,q)}$ via the place-values of positional numeral systems and used this stratification to deduce numeral systems for which the associated $h^\ast$-polynomials $h^\ast(\Delta_{(1,q)};z)$ were all real-rooted.  
In the following, we will use Theorem~\ref{thm: box polynomial of q} to show that these properties of the $h^\ast$-polynomial are again inherited by the associated local $h^\ast$-polynomials.  
We first recall the necessary details from~\cite{S17} in subsections~\ref{subsec: numeral systems and weighted projective spaces} and~\ref{subsec: real-rootedness}.

\subsection{Numeral Systems and Weighted Projective Spaces}
\label{subsec: numeral systems and weighted projective spaces}
The {\bf normalized volume} of a lattice polytope is the value given by evaluating its $h^\ast$-polynomial at $1$.  
In this paper, the normalized volume of the simplex $\Delta_{(1,q)}$ is denoted $Q$.  
When studying properties of lattice polytopes it can often be useful to stratify the collection of all lattice polytopes by their normalized volume rather than their dimension.  
In special cases, this can even allow us to discover recursive formulae that can be used to prove desirable properties of the associated $h^\ast$-polynomials, such as real-rootedness.  
To do exactly this, in~\cite{S17} the author studied when one simplex $\Delta_{(1,q)}$ in each dimension $n\geq1$ with normalized volume $Q_n$ can be found so that the sequence $(Q_n)_{n=0}^\infty$ are the place values of some (positional) numeral system.

A (positional) numeral system is a method for expressing numbers.
A {\bf numeral} is a string of nonnegative integers $\eta = \eta_{n-1}\eta_{n-2}\cdots\eta_0$ and the location of $\eta_i$ is called the {\bf place} of $i$.  
The {\bf digits} are the numbers allowable in the place $i$ and the $i^{th}$ {\bf base} is the number of digits for place $i$.  
A {\bf (positional) numeral system} is a sequence of positive integers $(a_n)_{n=0}^\infty$ satisfying $a_0 = 1<a_1<a_2<\cdots$, and the term $a_n$ is called its {\bf $n^{th}$ place value}.  
Any nonnegative integer $b$ can be expressed uniquely as a numeral with respect to the numeral system $(a_n)_{n=0}^\infty$ by repeatedly performing Euclidean division starting with the largest $a_{n-1}$ such that $b>a_{n-1}$ and repeating this process with the remainder (see for instance Theorem~1 in~\cite{F85}).  
Specifically, if $b = p_{n-1}a_{n-1}+r_{n-1}$ and $r_{n-1} = p_{n-2}a_{n-2}+r_{n-2}$ then $p_{n-1}$ is the digit in the $(n-1)^{st}$ place and $p_{n-2}$ is the digit in the $(n-2)^{nd}$ place in the numeral representing $b$, and so on.
\begin{example}
\label{ex: two numeral systems}
The following numeral systems will be of special interest in the coming sections.
\begin{enumerate}
	\item (The Binary Numeral System). Let $(a_n)_{n=0}^\infty := (2^n)_{n=0}^\infty$.  Then the $n^{th}$ place value is $2^n$ and the possible digits in each place are only $0$ or $1$.  The number $13$ is represented by the numeral $1101$ since
	\[
	13 = 1\cdot2^3+1\cdot2^2+0\cdot2^1+1\cdot2^0.
	\]
	\item (The Base-$r$ Numeral System).  Generalizing the previous example, let $r\geq 2$ and $(a_n)_{n=0}^\infty := (r^n)_{n=0}^\infty$.  Then the $n^{th}$ place value is $r^n$ and the possible digits in each place are $[r-1]_0$.  For example, when $r=10$, we get the typical base-$10$ numeral system which we all use every day.  
	\item (The Factoradic Numeral System). Let  $(a_n)_{n=0}^\infty := ((n+1)!)_{n=0}^\infty$.  Then the $n^{th}$ place value is $(n+1)!$ and the possible digits in each place are only $[n]_0$.  This numeral system is used in computer science to efficiently encode permutations in a lexicographic order, and it will be the focus of Section~\ref{sec: the factoradic simplex}.  
\end{enumerate}
\end{example}

Thinking geometrically, given a numeral system $(a_n)_{n=0}^\infty$ we can ask for a sequence of $\Delta_{(1,q)}$, one in each dimension $n\geq1$, with normalized volume $Q_n$ such that $Q_n = a_n$ for all $n\geq 1$.  
As demonstrated in \cite{S17}, if $\Delta_{(1,q)}$ are all well-chosen, then the $h^\ast$-polynomials of these simplices can be nicely expressed in terms of the combinatorics of the representations of integers in the numeral system $(a_n)_{n=0}^\infty$.  
Even more, these expressions can help us recover real-rootedness of the $h^\ast$-polynomials by revealing recursions that preserve interlacing of polynomials.
We now summarize the theory we will need on real-rootedness and interlacing polynomials in the next subsection.

\subsection{Real-rooted and Interlacing Polynomials}
\label{subsec: real-rootedness}
Let $p(z) = p_0+p_1z+\cdots+p_dz^d\in\R[z]$ be a real-rooted polynomial with degree $d$, denoted $\deg(p) = d$, and let $\alpha_1\geq \alpha_2\geq\cdots\geq \alpha_d$ denote its roots.  
If $q(z) = q_0+q_1z+\cdots+q_nz^n\in\R[z]$ is a second real-rooted polynomial with $\deg(q) = n$ and roots $\beta_1\geq \beta_2\geq\cdots\geq \beta_n$.  
We say that $q$ {\bf interlaces} $p$, denoted $q\preceq p$, if the roots of $p$ and $q$ can be ordered such that
\[
\alpha_1\geq\beta_1\geq \alpha_2\geq \beta_2\geq \alpha_3\geq \beta_3\geq\cdots.
\]
A sequence of polynomials $(f_i)_{i=0}^{m}$ with $f_i\in\R[z]$ of polynomials is called {\bf interlacing} if $f_i\preceq f_j$ for all $0\leq i\leq j\leq m$.  
Let $\mathcal{F}_m^+$ denote the space of all interlacing sequences $(f_i)_{i=0}^{m}$  for which $f_i$ has only nonnegative coefficients for all $0\leq i\leq m$.  
In~\cite{B15}, the author gave a complete characterization of matrices $(G_{i,j}(z))_{i,j=1}^{m,n}$ that map $\mathcal{F}_n^+$ to $\mathcal{F}_m^+$.  
These matrices are said to {\bf preserve interlacing}, and by iteratively applying the same matrix we produce recursions that preserve interlacing. 
In the coming sections, we will use two recursions that preserve interlacing.
\begin{lemma}
\cite[Corollary~8.7]{B15}
\label{lem: inversion sequence recursion}
Let $(f_i)_{i=0}^{n}$ be an interlacing sequence of polynomials.  
For $m\in\Z_{\geq0}$ and $\varphi:[m]_0\longrightarrow\Z_{\geq0}$ satisfying $\varphi(i)\leq\varphi(i+1)$ for all $i\in[m-1]_0$, define the polynomials 
\[
g_i:=z\sum_{j<\varphi(i)}f_i+\sum_{\varphi(i)\geq j}f_i.
\]
Then the sequence of polynomials $(g_i)_{i=0}^m$ is interlacing.
\end{lemma}
A closely related, but distinct, recursion that also preserves interlacing is the following:
\begin{lemma}
\cite[Lemma~4.4]{S17}
\label{lem: extended inversion sequence recursion}
Let $(f_i)_{i=0}^{n}$ be an interlacing sequence of polynomials.  
For $i\in[n]_0$ define the polynomial
\[
g_i:=z\sum_{j\leq \varphi(i)}f_i+\sum_{\varphi(i)\geq j}f_i
\]
Then the sequence of polynomials $(g_i)_{i=0}^n$ is interlacing.
\end{lemma}
In~\cite{S17}, the above lemmas were used to prove that families of $\Delta_{(1,q)}$ associated via their normalized volumes to the factoradic and base-$r$ numeral systems described in Example~\ref{ex: two numeral systems} all have real-rooted Ehrhart $h^\ast$-polynomials.  
In the coming sections, we prove these families of simplices also have real-rooted local $h^\ast$-polynomials. 

\section{The Factoradic Simplex}
\label{sec: the factoradic simplex}
Given a permutation $\pi=\pi_1\pi_2\cdots\pi_n\in\mathfrak{S}_n$ we say that an index $i\in[n-1]$ is a {\bf descent} of $\pi$ if $\pi_i>\pi_{i+1}$.  
We then let
\[
\des(\pi) := \left|\{i\in[n-1] \, : \, \pi_i>\pi_{i+1}\}\right|,
\]
and we call the polynomial
\[
A_n(z) :=\sum_{\pi\in\mathfrak{S}_n}z^{\des(\pi)}
\]
the {\bf $n^{th}$ Eulerian polynomial}.
It is well-known that $A_n(z)$ is symmetric and real-rooted of degree $n-1$ for all $n\geq 1$.  
We additionally define
\[
\maxDes(\pi) := \max\{i\in[n-1] \, : \, \pi_i>\pi_{i+1}\}, 
\]
and we define a second polynomial
\[
B_n(z) :=\sum_{\pi\in\mathfrak{S}_n}z^{\maxDes(\pi)}.
\]
We further let $a_{n,k}$ and $b_{n,k}$ denote the $k^{th}$ coefficient of $A_n(z)$ and $B_n(z)$, respectively.
It follows from Lemma 3.4 in~\cite{S17} that $B_n(z)$ is also a unimodal polynomial of degree $n-1$, and in~\cite{S17}, the author used the simplices $\Delta_{(1,q)}$ to offer a geometric interpolation between the two generating polynomials $A_n(z)$ and $B_n(z)$.  
\begin{definition}
\label{def: factoradic simplex}
For $n\geq 1$, let $q^{(n)} :=(b_{n+1,1},b_{n+1,2},\ldots,b_{n+1,n})$ be the coefficient sequence of $B_{n+1}(z)$ excluding its constant term.  
The {\bf factoradic $n$-simplex} is the $n$-simplex
\[
\Delta_n^! :=\Delta_{\left(1,q^{(n)}\right)}.
\]
\end{definition}
In~\cite{S17} the author proved the following theorem which shows that if one takes the coefficient vector of the polynomial $B_n(z)$ as the vector $q$, then the polynomial $A_n(z)$ is returned as the $h^\ast$-polynomial of the simplex $\Delta_{(1,q)}$.   
\begin{theorem}
\cite[Theorem~3.5]{S17}
\label{thm: h*-polynomial of factoradic simplex}
The factoradic $n$-simplex $\Delta_n^!$ has $h^\ast$-polynomial
\[
h^\ast(\Delta_n^!;z) = A_{n+1}(z).
\]
\end{theorem}
Since Theorem~\ref{thm: h*-polynomial of factoradic simplex} immediately implies that $h^\ast(\Delta_n^!;z)$ is real-rooted and unimodal, it is natural to ask if these same properties hold for the associated local $h^\ast$-polynomial $\ell^\ast(\Delta_n^!;z)$.  
Indeed, we will prove that these properties do in fact carry over to $\ell^\ast(\Delta_n^!;z)$.  
To do this, we must first identify a nice combinatorial expression for this polynomial in terms of descent statistics for permutations.  
This will require a few more definitions, which will in turn shed a bit more light on the choice of the name ``factoradic $n$-simplex.''

Given a permutation $\pi\in\mathfrak{S}_n$, its associated {\bf Lehmer code} (or {\bf inversion sequence}) is the sequence $\ell(\pi) :=(\ell_{n-1},\ell_{n-2},\ldots, \ell_1)$, where
\[
\ell_i := \left|\{
0\leq j< i \, : \, \pi_{n-i}>\pi_{n-j}
\}\right|.
\]
The collection of all Lehmer codes for $\mathfrak{S}_n$ consists of the collection of lattice points
\[
L_n:=\left\{
(\ell_{n-1},\ell_{n-2},\ldots, \ell_1)\in\Z^{n-1} 
\, : \,
\ell_k\in[k]_0,\, k\in[n-1]
\right\},
\]
and there is a natural bijection between $\mathfrak{S}_n$ and $L_n$.  
In particular, to construct a permutation $\pi$ from its Lehmer code start with the zero-indexed sequence of numbers $s = (s_i)_{i=0}^{n-1} := (1,2,\ldots,n)$.  
Beginning with $k = 1$, set $\pi_k := s_{\ell_{n-k}}$, drop $s_{\ell_{n-k}}$ from the list $s$ and repeat with $k:=k+1$.  
Finally, end by setting $\pi_n$ equal to the last remaining number.  

It turns out that Lehmer codes offer computer scientists an efficient way to encode the permutations $\mathfrak{S}_n$ in a lexicographic order.  
Given two strings of positive integers $a := a_1a_2\cdots a_n$ and $b:= b_1b_2\cdots b_n$, we say that $a$ is {\bf lexicographically larger} than $b$ if and only if the leftmost nonzero number in the string $(a_1-b_1)(a_2-b_2)\cdots(a_n-b_n)$ is positive.  
To each Lehmer code $\ell:=(\ell_{n-1},\ell_{n-2},\ldots, \ell_1)\in L_n$, we associate a unique number $0\leq b^{(\ell)}<n!$
\begin{equation}
\label{eqn: factoradic number}
b^{(\ell)} :=\sum_{i=1}^{n-1}\ell_{n-i}(n-i)!, 
\end{equation}
and we note that every such number $0\leq b< n!$ has such a unique representation.  
The Lehmer code $\ell$ associated to $b$ via equation~\eqref{eqn: factoradic number} is the representation of $b$ in the factoradic numeral system, which was introduced in Example~\ref{ex: two numeral systems}.  
Since $0\leq b^{(m)}<b^{(\ell)}<n!$ if and only if $\ell$ is lexicographically larger than $m$, this provides us with a natural lexicographic ordering of the permutations via their Lehmer codes. 
In particular, given $0\leq b<n!$, we let $\pi^{(b)}$ denote the associated $b^{th}$ lexicographically largest permutation in $\mathfrak{S}_n$.  

Similar to the permutations $\pi \in \mathfrak{S}_n$, we can also define descents in Lehmer codes: assuming that $\ell_0:=0$, we say that an index $i\in[n-1]$ is a {\bf descent} of a Lehmer code $\ell:=(\ell_{n-1},\ell_{n-2},\ldots, \ell_1)$ if $\ell_i>\ell_{i-1}$, and we let 
\[
\des(\ell) := \left|\{i\in[n-1] \, : \, \ell_i>\ell_{i-1}\}\right|.  
\]
It is a short exercise to check that $n-k$ is a descent of $\ell(\pi)$ if and only if $k$ is a descent in $\pi$.
In order to prove Theorem~\ref{thm: h*-polynomial of factoradic simplex}, in~\cite{S17} the author showed that, in the context of Theorem 2.2 of~\cite{BDS16} and $\Delta_n^!$, for all $0\leq b< n!$,
\[
\omega(b) = \des(\pi^{(b)}).
\]
That is, in Theorem~\ref{thm: h*-polynomial of factoradic simplex}, the Eulerian polynomial $A_n(z)$ is computed with respect to the lexicographic ordering of the permutations.   
By Theorem~\ref{thm: box polynomial of q}, to compute the local $h^\ast$-polynomial of $\Delta_n^!$, we then need only compute which values of $0\leq b<n!$ satisfy the necessary divisibility conditions with respect to $q^{(n)} :=(b_{n+1,1},b_{n+1,2},\ldots,b_{n+1,n})$.  
As we will see in the next theorem, these are, intriguingly, exactly those values $0\leq b<n!$ that are congruent to 1 or 5 modulo 6.
In the following, we let $\overline{k}_m$ denote the equivalence class of all integers congruent to $k$ modulo $m$. 
\begin{theorem}
\label{thm: factoradic box polynomial}
The factoradic $n$-simplex $\Delta_n^!$ has local $h^\ast$-polynomial
\[
\ell^\ast(\Delta_n^!;z) = \sum_{b\in\left[(n+1)!\right]\cap\left(\overline{1}_6\cup\overline{5}_6\right)}z^{\des(\pi^{(b)})}.
\]
In particular, 
\[
\ell^\ast(\Delta_n^!;1) = \frac{n!}{3}.
\]
\end{theorem}

\begin{proof}
To prove the statement we apply Theorem~\ref{thm: box polynomial of q}.
In particular, for $q = (b_{n+1,1},b_{n+1,2},\ldots,b_{n+1,n})$, we would like to show that 
\[
T_q = \left[n+1\right]\cap\left(\overline{1}_6\cup\overline{5}_6\right).
\]
To see this, note that by Lemma 3.4 of~\cite{S17}, $0< b<(n+1)!$ satisfies $b\in T_q$ if and only if $(n+1)!$ does not divide 
\[
\frac{(n+1)!b}{(n-k+1)!+(n-k)!}
\]
for all $k\in[n]$.  
Equivalently, $b\in T_q$ if and only if $(n-k+1)!+(n-k)!$ does not divide $b$ for all $k\in[n]$.  
To see that the latter condition is equivalent to $b\in\overline{1}_6\cup\overline{5}_6$, notice that when $k = n$ we have 
\[
(n-k+1)!+(n-k)! = 1!+0! = 2,
\]
and when $k = n-1$ we have that 
\[
(n-k+1)!+(n-k)! = 2!+1! = 3.  
\]
Moreover, for every $k<n-1$, both $2$ and $3$ divide $(n-k+1)!+(n-k)!$.  
Thus, $b\in T_q$ if and only if it is never divisible by $2$ or $3$, and this happens only if $b\in\overline{1}_6\cup\overline{5}_6$.  
Since, by the proof of Theorem 3.5 in~\cite{S17}, we know that $\omega(b)  = \des(\pi^{(b)})$ for all $0\leq b<(n+1)!$, then we can apply Theorem~\ref{thm: box polynomial of q} to recover that
\[
\ell^\ast(\Delta_n^!;z) = \sum_{b\in\left[(n+1)!\right]\cap\left(\overline{1}_6\cup\overline{5}_6\right)}z^{\des(\pi^{(b)})}.
\]
To see the final statement, notice that $\ell^\ast(\Delta_n^!;1)$ is the number of integers $0\leq b<n!$ that are congruent to $1$ or $5$ modulo $6$, and this can be computed as
\[
n!-4\left(\frac{n!}{6}\right) = \frac{n!}{3}.
\]
This completes the proof.
\end{proof}

\begin{remark}
\label{rmk: factoradics at 1}
As noted in Theorem~\ref{thm: factoradic box polynomial}, $\ell^\ast(\Delta_n^!;1) = \frac{n!}{3}$ and thus the open parallelpiped $\Pi_!^\circ$ of $\Delta_n^!$ contains one third of the lattice points within the half-open parallelpiped $\Pi_!$ of $\Delta_n^!$.  
Since the local $h^\ast$-polynomial $\Delta_n^!$ enumerates these lattice points via the descent statistics of their associated permutations, it is interesting to note that the sequence $\left(\frac{n!}{3}\right)_{n=0}^\infty$ also has other known connections to permutation statistics.  
This sequence is number A002301 on the \emph{Online Encyclopedia of Integer Sequences} (\emph{OEIS}) \cite{OEIS}, where more details of these connections can be found.
\end{remark}

\begin{figure}[t!]
\begin{tikzpicture}[scale=0.5]
 	 \node at (0,0) {$1$};
 	 \node at (-1,-1) {$1$};
 	 \node at (1,-1) {$1$};
 	 \node at (-2,-2) {$1$};
 	 \node at (0,-2) {$6$};
 	 \node at (2,-2) {$1$};
 	 \node at (-3,-3) {$1$};
 	 \node at (-1,-3) {$19$};
 	 \node at (1,-3) {$19$};
 	 \node at (3,-3) {$1$};
 	 \node at (-4,-4) {$1$};
 	 \node at (-2,-4) {$48$};
 	 \node at (0,-4) {$142$};
 	 \node at (2,-4) {$48$};
 	 \node at (4,-4) {$1$};
	 \node at (-5,-5) {$1$};
 	 \node at (-3,-5) {$109$};
 	 \node at (-1,-5) {$730$};
 	 \node at (1,-5) {$730$};
 	 \node at (3,-5) {$109$};
 	 \node at (5,-5) {$1$};
	 \node at (-6,-6) {$1$};
	 \node at (-4,-6) {$234$};
 	 \node at (-2,-6) {$3087$};
 	 \node at (0,-6) {$6796$};
 	 \node at (2,-6) {$3087$};
 	 \node at (4,-6) {$234$};
 	 \node at (6,-6) {$1$};
\end{tikzpicture}
\centering
\caption{The triangle of coefficients for $\ell^\ast(\Delta_n^!;z)$ for $2\leq n\leq 8$.}
\label{fig: triangle}
\end{figure}

\begin{example}
[Some Factoradic local $h^\ast$-Polynomials]
\label{ex: some factoradic polynomials}
The triangle of coefficients of $\ell^\ast(\Delta_n^!;z)$ for $2\leq n\leq 8$ is given in Figure~\ref{fig: triangle}.  
We see that these are indeed not the {\bf derangement polynomials}, which were shown to be the local $h^\ast$-polynomials of the $s$-lecture hall simplices $P_n^s$ for $s = (2,3,\ldots,n)$ \cite{GS18}.
The simplices $\Delta_n^!$ and $P_n^s$ are both {\bf reflexive} (see~\cite{S17} for the definition of reflexive) with $h^\ast$-polynomial $A_{n+1}(z)$, and the combinatorics of both are intimately tied to that of inversion sequences \cite{SV15,S17}.  
Despite this, the distinction in their local $h^\ast$-polynomials further highlights their fundamentally different geometry and combinatorics. 
\end{example}

Using the formula from Theorem~\ref{thm: factoradic box polynomial} we can now affirmatively answer Question~\ref{quest: box polynomials} for the factoradic simplex.

\begin{theorem}
\label{thm: factoradic real-rootedness}
The local $h^\ast$-polynomial $\ell^\ast(\Delta_n^!;z)$ of the factoradic $n$-simplex $\Delta_n^!$ is real-rooted, and thus unimodal.  
\end{theorem}

\begin{proof}
To prove that the polynomials $\ell^\ast(\Delta_n^!;z)$ are real-rooted we will identify a polynomial recursion that will allow us to apply Lemma~\ref{lem: inversion sequence recursion}.
For convenience, we first recall that the descents in a permutation $\pi\in\mathfrak{S}_n$ are in bijection with the descents in the associated Lehmer code $\ell(\pi)$.  
Thus, we may also express the polynomial $\ell^\ast(\Delta_n^!;z)$ as
\[
\ell^\ast(\Delta_n^!;z) = 
\sum_{b\in\left[(n+1)!\right]\cap\left(\overline{1}_6\cup\overline{5}_6\right)}z^{\des(\ell(b))},
\]
where $\ell(b) = (\ell_{b,n},\ell_{b,n-1},\ldots,\ell_{b,1}) :=\ell(\pi^{(b)})$ for all $0\leq b< (n+1)!$.  
Recall that $0\leq \ell_{b,n}<n$, and let $\chi(p)$ denote the boolean function yielding $1$ if $p$ is a true statement and $0$ otherwise. 
For $0\leq k< n$, we then define the polynomial
\begin{equation}
\label{eqn: refined box}
B_{n,k}(z) :=\sum_{b\in\left[(n+1)!\right]\cap\left(\overline{1}_6\cup\overline{5}_6\right)}\chi\left(\ell_{b,n} = k\right)z^{\des(\ell(b))}.
\end{equation}
Notice first that
\[
\ell^\ast(\Delta_n^!;z) = \sum_{k=0}^nB_{n,k}(z),
\quad
\mbox{and that}
\quad
\ell^\ast(\Delta_n^!;z) = B_{n+1,0}(z).
\]
Moreover, it is quick to check that we have the recursion
\begin{equation}
\label{eqn: factoradic recursion}
B_{n,k}(z) = z\sum_{t<k}B_{n-1,t}(z)+\sum_{t\geq k}B_{n-1,t}(z), 
\end{equation}
with initial conditions $B_{3,0}(z) = z$, $B_{3,1}(z) = 0$, and $B_{3,2}(z) = z^2$.  
We note here that we must take the initial conditions with $n =3$ since if we express $b$ as in equation~\eqref{eqn: factoradic number}
\[
b = \sum_{i=1}^{n-1}\ell_{n-i}(n-i)!, 
\]
and set $b^\prime := \sum_{i=1}^{n-2}\ell_{n-i}(n-i)!,$ then we require that  $b$ is congruent to $1$ or $5$ modulo $6$ if and only if $b^\prime$ is congruent to $1$ or $5$ modulo $6$.  
However, this fails for $n<3$.  
(Note however, that $\ell^\ast(\Delta_1^!;z) = 0$ and $\ell^\ast(\Delta_2^!;z) =z$, so real-rootedness will still hold in all cases.)
Given the recursion in equation~\ref{eqn: factoradic recursion}, the local $h^\ast$-polynomial of the factoradic $n$-simplex $\Delta_n^!$ is then real-rooted, and thus unimodal, since we observed that $\ell^\ast(\Delta_n^!;z) = B_{n+1,0}(z)$.  
\end{proof}

\begin{remark}
\label{rmk: on the recursion used}
The same type of recursion as given in equation~\eqref{eqn: factoradic recursion} has recently been used in order to recover real-rootedness results for a number of polynomial generating functions in combinatorics \cite{BL16,GS18,SV15,S17}.  
In particular, this recursion as applied to refined descent polynomials over inversion sequences, defined analogously to those in equation~\eqref{eqn: refined box}, can be viewed as the ``reverse'' of the recursion applied to inversion sequences to recover real-rootedness of the Ehrhart $h^\ast$-polynomials of the $s$-lecture hall simplices \cite{SV15}.
\end{remark}

Since the local $h^\ast$-polynomials $\ell^\ast(\Delta_n^!;z)$ are both symmetric and real-rooted, it follows that they also have only nonnegative integer coefficients in another well-studied polynomial basis.
A polynomial $p(z)\in\R[z]$ of degree at most $n$ is called {\bf $\gamma$-nonnegative} (or {\bf $\gamma$-positive}) if it has only nonnegative coefficients when expressed in the basis $\left\{z^i(z+1)^{n-2i}\right\}_{i=0}^{\lfloor n/2\rfloor}$.
It is well-known that if a $p(z)$ is real-rooted and symmetric with nonnegative coefficients in the standard basis, then $p(z)$ is $\gamma$-nonnegative (see for instance Remark 3.1 in~\cite{B15}).  
Thus, we have the following corollary.

\begin{corollary}
\label{cor: factoradic gamma-nonnegative}
The local $h^\ast$-polynomial $\ell^\ast(\Delta_n^!;z)$ of the factoradic $n$-simplex $\Delta_n^!$ is $\gamma$-nonnegative.  
\end{corollary}

\begin{proof}
Since $\ell^\ast(\Delta_n^!;z)$ is symmetric and real-rooted with nonnegative coefficients it follows that it is also $\gamma$-nonnegative.
\end{proof}

\begin{remark}
\label{rmk: factoradic gamma-nonnegativity}
In Corollary 3.3 of~\cite{B15}, the author presents a sufficient condition for a restricted Eulerian polynomial
\[
A(T;z) := \sum_{\pi\in T\subset \mathfrak{S}_n}z^{\des(\pi)}
\]
to be $\gamma$-nonnegative.  
This sufficient condition requires that the subset $T\subset\mathfrak{S}_n$ be invariant under the {\bf valley--hopping} action of $\Z_2^n$ on $\mathfrak{S}_n$ first considered in~\cite{FS73}.  
Since the polynomial $\ell^\ast(\Delta_n^!;z)$ is an example of restricted Eulerian polynomial (as seen by Theorem~\ref{thm: factoradic box polynomial}), it is natural to ask whether or not its $\gamma$-nonnegativity arises as a consequence of Corollary 3.3 from~\cite{B15}.  
In fact, it does not.  
To see this, suppose that 
\[
T := 
\left\{
\pi^{(b)}\in\mathfrak{S}_n
\, : \,
b\in\overline{1}_6\cup\overline{5}_6
\right\}, 
\]
and notice that $\pi^{((n+1)!-1)} = n\cdots321\in T$.  
The unique element of the orbit of $\pi^{((n+1)!-1)}$ under the valley--hopping action that has no double descents (i.e.~no indices $k\in[n]$ for which $\pi_{k-1}>\pi_k>\pi_{k+1}$) is $\pi^{(0)} = 123\cdots n$.  
However, $\pi^{(0)}\notin T$, and thus $T$ cannot be invariant under the valley--hopping action.  

Since the valley--hopping action cannot be used to give a direct combinatorial proof of the $\gamma$-nonnegativity of $\ell^\ast(\Delta_n^!;z)$ that is implied by its real-rootedness, it would therefore be of  interest to provide an alternate combinatorial interpretation of the coefficients of $\ell^\ast(\Delta_n^!;z)$ when it is written in the $\gamma$-basis.
\end{remark}

\section{The Base-$r$ Simplex}
\label{sec: the base-r simplex}

A second family of $\Delta_{(1,q)}$'s that are known to have real-rooted $h^\ast$-polynomials are those with 
\[
q^{(r,n)}:=\left((r-1),(r-1)r,\ldots,(r-1)r^{i-1},\ldots,(r-1)r^{n-1}\right),
\]
for $r\geq 2$ and $n\geq 1$ \cite{S17}.  
For simplicity, we will set $\mathcal{B}_{(r,n)} :=\Delta_{(1,q^{(r,n)})}$ throughout this section.  
In~\cite{S17}, these simplices are referred to as the {\bf base-$r$ $n$-simplices} since the normalized volume of $\mathcal{B}_{(r,n)}$ is equal to $r^n$, and this is the $n^{th}$ place value in the base-$r$ numeral system. 

In~\cite{GS18} the authors observed that the $s$-lecture hall simplices, which have real-rooted $h^\ast$-polynomials also have real-rooted local $h^\ast$-polynomials.  
In~\cite{S17} and Section~\ref{sec: the factoradic simplex} we, respectively, saw that the same holds true for the factoradic simplex $\Delta_n^!$.  
In this section, we continue this trend by further observing that all base-$r$ simplices have real-rooted local $h^\ast$-polynomials as well.  
To do this, we must first define an additional family of polynomials.  

For $r\geq2$ and $n\geq 1$ define the polynomial
\[
f_{(r,n)}(z) := (1+z+z^2+\cdots+z^{r-1})^n.
\]
By writing each $0\leq m\leq n(r-1)$ as $m = p(r-1)+q$ where $0\leq q<r-1$, the uniqueness of this expression guarantees that for any $r\geq2$ and polynomial $f(z)\in\R[z]$ there exist unique polynomials $f^{(0)},\ldots,f^{(r-2)}\in\R[z]$ such that
\[
f(z) = f^{(0)}(z^{r-1})+zf^{(1)}(z^{r-1})+z^2f^{(2)}(z^{r-1})+\cdots+z^{r-2}f^{(r-2)}(z^{r-1}).
\]
For $\ell \in[r-2]_0$ we then let $f^{\langle r-1,\ell\rangle} := f^{(\ell)}$.  
In~\cite{S17}, the author gave the following formula for $h^\ast(\mathcal{B}_{(r,n)};z)$, from which they recovered that this $h^\ast$-polynomial is real-rooted.
\begin{theorem}
\cite[Theorem~4.2]{S17}
\label{thm: base-r h*-polynomial}
For $r\geq 2$ and $n\geq 1$, the base-$r$ $n$-simplex $\mathcal{B}_{(r,n)}$ has $h^\ast$-polynomial
\[
h^\ast(\mathcal{B}_{(r,n)};z) = f^{\langle r-1,0\rangle}_{(r,n)} +z\sum_{\ell=1}^{r-2}f^{\langle r-1,\ell\rangle}_{(r,n)}.
\]
\end{theorem}

In a similar fashion, one can show that the local $h^\ast$-polynomial can also be computed in terms of the polynomials $f^{\langle r-1,\ell\rangle}_{(r,n)}$.
\begin{theorem}
\label{thm: base-r box polynomial}
For $r\geq2$ and $n\geq 1$, the base-$r$ $n$-simplex has local $h^\ast$-polynomial 
\[
\ell^\ast(\mathcal{B}_{(r,n)};z) = z\sum_{i=0}^{r-2}f^{\langle r-1,i\rangle}_{(r,n-1)}+z\sum_{\ell=1}^{r-2}\left(\sum_{i=0}^{\ell-1}f^{\langle r-1,i\rangle}_{(r,n-1)}+z\sum_{i=\ell}^{r-2}f^{\langle r-1,i\rangle}_{(r,n-1)}\right).
\]
\end{theorem}

\begin{proof}
To deduce the desired formula for $\ell^\ast(\mathcal{B}_{(r,n)};z)$ we first apply Theorem~\ref{thm: box polynomial of q}.  
We have that $Q = r^n$, and so each lattice point in the half-open parallelpiped for $\mathcal{B}_{(r,n)}$ corresponds to an integer $0\leq b<r^n$.  
As such, we will denote the lattice point corresponding to $b$ by $x^{(b)}$. 
Letting $\Pi_{(r,n)}^\circ$ denote the open parallelpiped of $\mathcal{B}_{(r,n)}$, Theorem~\ref{thm: box polynomial of q} tells us that $x^{(b)}\in\Pi_{(r,n)}^\circ$ if and only if $r^n \nmid \,(r-1)r^{i-1}b$ for all $i\in[n]$.  
Equivalently, $x^{(b)}\in\Pi_{(r,n)}^\circ$ if and only if $r^{n-i+1} \,\nmid (r-1)b$ for all $i\in[n]$. 
Furthermore, since $r$ and $r-1$ are always coprime, then so are $r^{n-i+1}$ and $(r-1)$ for all $i\in[n]$.  
Therefore, $x^{(b)}\in\Pi_{(r,n)}^\circ$ if and only if $r^{n-i+1} \,\nmid b$ for all $i\in[n]$. 
Equivalently, we then have that $x^{(b)}\in\Pi_{(r,n)}^\circ$ if and only if $r\nmid \, b$. 
It follows that,
\begin{equation}
\label{eqn: h* subtraction}
\ell^\ast(\mathcal{B}_{(r,n)};z) = h^\ast(\mathcal{B}_{(r,n)};z) - \sum_{0<b<r^n\, : \, r\nmid\, b}z^{\omega(b)}.
\end{equation}
To compute the latter sum in equation~\eqref{eqn: h* subtraction}, notice that for $0<b<r^n$, if $r\big| b$ such that $b = rb^\prime$ for $0<b^\prime<r^{n-1}$ then 
\begin{equation*}
\begin{split}
\omega(b)
&= b - \sum_{i = 1}^n\left\lfloor\frac{(r-1)b}{r^i}\right\rfloor,\\
&= rb^\prime -(r-1)b^\prime - \sum_{i = 1}^{n-1}\left\lfloor\frac{(r-1)b^\prime}{r^i}\right\rfloor,\\
&= b^\prime - \sum_{i = 1}^{n-1}\left\lfloor\frac{(r-1)b^\prime}{r^i}\right\rfloor,\\
&= \omega(b^\prime).
\end{split}
\end{equation*}
Therefore, we have that the latter sum in equation~\eqref{eqn: h* subtraction} is actually the $h^\ast$-polynomial of $\mathcal{B}_{(r,n-1)}$, the base-$r$ simplex of one dimension less. 
That is,
\[
\sum_{0<b<r^n\, : \, r\nmid\, b}z^{\omega(b)} = h^\ast(\mathcal{B}_{(r,n-1)};z).
\]
Consequently, we have that 
\begin{equation}
\label{eqn: dimension subtraction}
\ell^\ast(\mathcal{B}_{(r,n)};z) = h^\ast(\mathcal{B}_{(r,n)};z) - h^\ast(\mathcal{B}_{(r,n-1)};z).
\end{equation}
In the proof of Theorem 4.2 in~\cite{S17}, it is shown that the polynomials $f^{\langle r-1,\ell\rangle}_{(r,n)}$ for $0\leq \ell\leq r-2$ satisfy the recursion
\begin{equation}
\label{eqn: base-r recursion}
f^{\langle r-1,\ell\rangle}_{(r,n)} = \sum_{i=0}^\ell f^{\langle r-1,i\rangle}_{(r,n-1)}+z\sum_{i=\ell}^{r-2}f^{\langle r-1,i\rangle}_{(r,n-1)}.
\end{equation}
Using the recursion in equation~\eqref{eqn: base-r recursion} together with Theorem~\ref{thm: base-r h*-polynomial} we can deduce from equation~\eqref{eqn: dimension subtraction} that 
\begin{equation*}
\begin{split}
\ell^\ast(\mathcal{B}_{(r,n)};z) 
&= h^\ast(\mathcal{B}_{(r,n)};z) - h^\ast(\mathcal{B}_{(r,n-1)};z),\\
&= \left(f^{\langle r-1,0\rangle}_{(r,n)} - f^{\langle r-1,0\rangle}_{(r,n-1)}\right) + z\sum_{\ell=1}^{r-2}\left(f^{\langle r-1,\ell\rangle}_{(r,n)} - f^{\langle r-1,\ell\rangle}_{(r,n-1)}\right),\\
&= z\sum_{i=0}^{r-2}f^{\langle r-1,i\rangle}_{(r,n-1)}+z\sum_{\ell=1}^{r-2}\left(\sum_{i=0}^{\ell-1}f^{\langle r-1,i\rangle}_{(r,n-1)}+z\sum_{i=\ell}^{r-2}f^{\langle r-1,i\rangle}_{(r,n-1)}\right),
\end{split}
\end{equation*}
which completes the proof.
\end{proof}

It then follows from the formula given in Theorem~\ref{thm: base-r box polynomial} that $\ell^\ast(\mathcal{B}_{(r,n)};z)$ is also a real-rooted, and therefore unimodal, polynomial for all $r\geq 2$ and $n\geq 1$.  
\begin{corollary}
\label{cor: base-r real-rootedness}
The local $h^\ast$-polynomial $\ell^\ast(\mathcal{B}_{(r,n)};z)$ of the base-$r$ $n$-simplex is real-rooted, and thus unimodal, for all $r\geq2$ and $n\geq1$.
\end{corollary}

\begin{proof}
To prove that $\ell^\ast(\mathcal{B}_{(r,n)};z)$ is real-rooted we again apply the theory of interlacing sequences and the interlacing-preserving recursions presented in Lemmas~\ref{lem: inversion sequence recursion} and ~\ref{lem: extended inversion sequence recursion}. 
By Lemma~\ref{lem: extended inversion sequence recursion}, we know that the recursion for the polynomials $f^{\langle r-1,\ell\rangle}_{(r,n)}$ given in equation~\eqref{eqn: base-r recursion} preserves interlacing.  
Since the polynomials 
\[
g_{\ell-1}:=\sum_{i=0}^{\ell-1}f^{\langle r-1,i\rangle}_{(r,n-1)}+z\sum_{i=\ell}^{r-2}f^{\langle r-1,i\rangle}_{(r,n-1)}
\]
for $1\leq \ell \leq r-1$ are produced by applying the recursion in Lemma~\ref{lem: inversion sequence recursion} once to the interlacing sequence
\[
\left(f^{\langle r-1,r-2\rangle}_{(r,n-1)},\ldots,f^{\langle r-1,1\rangle}_{(r,n-1)},f^{\langle r-1,0\rangle}_{(r,n-1)}\right),
\]
and since this recursion also preserves interlacing, then we know that the sequence of polynomials
$
\left(g_{r-2},\ldots,g_1,g_0\right)
$
is also an interlacing sequence.
Moreover, by Theorem~\ref{thm: base-r box polynomial} we know that 
\[
\ell^\ast(\mathcal{B}_{(r,n)};z) = z\sum_{\ell = 0}^{r-1} g_{\ell-1}.
\]
It follows that $\ell^\ast(\mathcal{B}_{(r,n)};z)$ is real-rooted (by, for example, Lemma 2.2 in~\cite{B06}). 
\end{proof}

Similar to the case of the local $h^\ast$-polynomials for the factoradic simplices $\Delta_n^!$, the real-rootedness of $\ell^\ast(\mathcal{B}_{(r,n)};z)$ immediately implies that it is also $\gamma$-nonnegative.
\begin{corollary}
\label{cor: base-r gamma-nonnegativity}
For all $r\geq 2$ and $n\geq 1$ the local $h^\ast$-polynomial $\ell^\ast(\mathcal{B}_{(r,n)};z)$ of the base-$r$ $n$-simplex is $\gamma$-nonnegative.  
\end{corollary}

\begin{example}
[The Base-$2$ Simplex]
\label{ex: the base-2 simplex}
In the case when $r = 2$, the normalized volume of $\mathcal{B}_{(2,n)}$ is equal to $2^n$, the $n^{th}$ place value in the binary numeral system.  
In Theorem 3.6 of~\cite{S17} it is shown that the $h^\ast$-polynomial of $\mathcal{B}_{(2,n)}$ is a familiar symmetric, real-rooted and unimodal polynomial; namely, 
\[
h^\ast(\mathcal{B}_{(2,n)};z) = (1+z)^n.
\]
By applying Corollary~\ref{cor: base-r real-rootedness} we can similarly deduce that
\begin{equation}
\label{eqn: base-2}
\ell^\ast(\mathcal{B}_{(2,n)};z) = z(1+z)^{n-1}.
\end{equation}
In particular, when $r = 2$, we have that $f_{(2,n)} = (1+z)^n$, $r-2 = 0$, and
\[
f_{(2,n)}^{\langle 1,0\rangle} = (1+z)^n.
\]
Since $r-2 = 0$, the second sum in Theorem~\ref{thm: base-r box polynomial} is empty, which leaves us with the formula for $\ell^\ast(\mathcal{B}_{(2,n)};z)$ in equation~\eqref{eqn: base-2}.  

Finally, we remark that the formula for $\ell^\ast(\mathcal{B}_{(2,n)};z)$ given in equation~\eqref{eqn: base-2} also has a natural combinatorial interpretation. 
In Theorem 3.6 of~\cite{S17}, it is shown that 
\[
h^\ast(\mathcal{B}_{(2,n)};z) = \sum_{b=0}^{2^n-1}z^{\supp_2(b)},
\]
where $\supp_2(b)$ denotes the number of $1$'s in the binary representation of the integer $b$.  
Equation~\eqref{eqn: base-2} simply says that 
\[
\ell^\ast(\mathcal{B}_{(2,n)};z) = \sum_{0<b<2^n\,:\,\textrm{$b$ is odd}}z^{\supp_2(b)}.
\]
\end{example}

\section{Final Remarks}
\label{sec: final remarks}

We end with a few comments on the possible future directions of research relating to the results in this note. 
First, we make a remark on the general relationship between real-rootedness of $h^\ast$-polynomials of simplices and that of their local $h^\ast$-polynomials:
The results of~\cite{GS18} and those of this paper demonstrate that certain families of lattice simplices with real-rooted $h^\ast$-polynomials also admit real-rooted local $h^\ast$-polynomials.  
It is therefore natural to ask if real-rootedness of the $h^\ast$-polynomial implies real-rootedness of the local $h^\ast$-polynomial and/or vice versa.  
Using the observations in this paper, we can already show that the converse of this statement does not hold: 
By Corollary~\ref{cor: projective space}, we know that the $2$-simplex $\Delta_{(1,q)}$ with $q = (1,1)$ has a real-rooted local $h^\ast$-polynomial, but not a real-rooted $h^\ast$-polynomial.  
On the other hand, the author currently is not aware of a lattice simplex with a real-rooted $h^\ast$-polynomial but a non-real-rooted local $h^\ast$-polynomial.  
Examples of such simplices and an understanding of their lattice point combinatorics would be interesting to see.

As noted in Corollaries~\ref{cor: factoradic gamma-nonnegative} and~\ref{cor: base-r gamma-nonnegativity}, it follows immediately from the real-rootedness of the local $h^\ast$-polynomials studied in this note that they are also $\gamma$-nonnegative.  
The well-known technique of valley-hopping does not seem to offer a second proof of this fact (see Remark~\ref{rmk: factoradic gamma-nonnegativity} in the case of the factoradic simplex).  
Thus, it would be interesting to see a purely combinatorial proof of the $\gamma$-nonnegativity of the local $h^\ast$-polynomials studied in this paper, as well as those in~\cite{GS18}.  

It was also recently shown that the local $h^\ast$-polynomial of a simplex satisfies a monotonicity property with respect the local $h$-polynomials of its geometric subdivisions (see Remark 7.18 in~\cite{KS16}).  
In particular, if $\Delta$ is a lattice simplex with a lattice triangulation $T$, and $\ell_\Delta(T_\Delta;z)$ is the local $h$-polynomial of the (abstract) subdivision of $\Delta$ defined by $T$, then 
\[
\ell^\ast(\Delta;z)\geq \ell_\Delta(T_\Delta;z).
\]
Thus, it would be interesting to know if the local $h^\ast$-polynomials computed here can be realized as local $h$-polynomials for some (possibly only abstract) subdivision of a simplex.  
One way to accomplish this would be to prove that the simplices studied here all admit unimodular triangulations (see Example 7.19 in~\cite{KS16}).  

\smallskip

\noindent
{\bf Acknowledgements}. 
The author was supported by an NSF Mathematical Sciences Postdoctoral Research Fellowship (DMS - 1606407). 
This manuscript was prepared as a contribution to the conference proceedings of the \emph{$2018$ Workshop on Lattice Polytopes} at Osaka University.  
The author thanks the organizers, Dr. Takayuki Hibi and Dr. Akiyoshi Tsuchiya, for their support during this workshop.

%
%
\end{document}